\documentclass{amsart}
\usepackage[utf8]{inputenc}
\usepackage[english]{babel}
\usepackage{enumitem}
\usepackage{hyperref}

\title{Characterising quasi-isometries of the free group}

\usepackage{mathabx}

\usepackage{dsfont}

\newtheorem{theorem}{Theorem}[section]
\newtheorem{thm}{Theorem}[section]
\newtheorem{lemma}[theorem]{Lemma}

\newtheorem{prop-defn}[theorem]{Proposition-Definition}
\newtheorem{claim}{Claim}
\newtheorem*{theorem:main}{Main Theorem} 
\theoremstyle{definition} 

\newtheorem{definition}[theorem]{Definition}
\newtheorem{remark}[theorem]{Remark}

\newtheorem*{remarks*}{Remarks}

\usepackage{graphicx}
\usepackage{commath}
\usepackage{hyperref}
\usepackage{dsfont}

\usepackage{tikz}

\usepackage{amssymb}
\usepackage{mathtools} 
\usepackage{float}

\DeclarePairedDelimiter\floor{\lfloor}{\rfloor}
\DeclarePairedDelimiter\ceil{\lceil}{\rceil}

\author[Antoine Goldsborough]{Antoine Goldsborough}
	\address{Maxwell Institute and Department of Mathematics, Heriot-Watt University, Edinburgh, UK}
	\email{ag2017@hw.ac.uk}

\author{Stefanie Zbinden}
	\address{Maxwell Institute and Department of Mathematics, Heriot-Watt University, Edinburgh, UK}
	\email{sz2020@hw.ac.uk}

\begin{document}

\maketitle

\begin{abstract}
We introduce the notion of mixed subtree quasi-isometries, which are self quasi-isometries of regular trees built in a specific inductive way. We then show that any self quasi-isometry of a regular tree is at bounded distance from a mixed-subtree quasi-isometry. Since the free group is quasi-isometric to a regular tree, this provides a way to describe all self quasi-isometries of the free group. In doing this, we also give a way of constructing quasi-isometries of the free group.  
\end{abstract}

\section{Introduction}

Quasi-isometries are the most fundamental maps in geometric group theory. However, for most metric spaces, very little is known about their quasi-isometry group and there are no known tangible ways to describe all quasi-isometries, except in some cases where quasi-isometric rigidity is known. Notable exceptions to this are Baumslag-Solitar groups which are described in \cite{WhyteBaumslag} and 3-dimensional solvable Lie groups which have been studied by Eskin, Fisher and Whyte in \cite{EskinFisherWhyte07,EskinFisherWhyte, EskinFisherWhyte2}. 

With this paper, we add the free group $\mathbb{F}_2$, or more generally regular trees, to the list of spaces where all quasi-isometries up to bounded distance can be described. 
In particular, we introduce the notion of a $D$-mixed subtree quasi-isometry which is a type of quasi-isometry from regular trees to themselves. While a precise definition can be found in Section \ref{sec:qi_of_reg_trees}, the main idea behind them is the following; having defined the quasi-isometry for vertices $v$ at distance $nD$ from the root, one next defines what the quasi-isometry does on the next level, that is, vertices at distance $(n+1)D$ from the root. Moreover, the valid choices of extending the map to the vertices at distance $(n+1)D$ only depend on which of the vertices of distance $nD$ are mapped to the same vertex, but is otherwise independent of the choices made previously. 

Our main theorem below states that a map from a regular tree to itself is a quasi-isometry if and only if it is at bounded distance from a mixed-subtree quasi-isometry.

\begin{thm}\label{thm:main}
     Let $T$ be a regular tree of degree at least 3, rooted at $v_0$. Let $f: T \to T$ be a $C$-quasi-isometry such that $f(v_0)=v_0$. Then there is a constant $D$ only depending on $C$ and a $D$-deep mixed subtree quasi-isometry $g: T\to T$ such that $f$ and $g$ are at bounded distance from each other. 
\end{thm}

Since regular trees of degree at least 3 and non-elementary free groups are quasi-isometric, the theorem above describes quasi-isometries of the free group $\mathbb F_2.$

Thanks to this independence mentioned above, mixed-subtree quasi-isometries are a useful tool to construct quasi-isometries with certain desired properties. For example, this technique was used in \cite{goldsboroughzbinden}, where the authors build a self quasi-isometry of $\mathbb F_2$ with the property that the push-forward of a simple random walk by this quasi-isometry does not have a well-defined drift.

We suspect that there might be other applications of this construction. For instance, one might want to consider `random quasi-isometries' of $\mathbb F_2$ and properties of a `generic' quasi-isometry. Further, this characterisation might allow one to better understand the quasi-isometry group $QI(\mathbb F_2)$.

\smallskip
\textbf{Outline.} In Section \ref{sec:prelim} we introduce the relevant notation and prove some of the technical results about quasi-isometries of trees. In particular, we extend a result of \cite{Nairne} and show that any quasi-isometry is at bounded distance from an order-preserving quasi-isometry. In Section \ref{sec:qi_of_reg_trees} we describe mixed-subtree quasi-isometries and prove Theorem \ref{thm:main}, which states that a map from a rooted tree of degree at least 3 to itself is a quasi-isometry if and only if it is at bounded distance from a mixed-subtree quasi-isometry.\\

\smallskip
\textbf{Acknowledgments.} We would like to thank Oli Jones, Alice Kerr, Patrick Nairne and our supervisor Alessandro Sisto for helpful discussions and feedback. We also thank the anonymous referee for very helpful comments which greatly improved the readability of the paper.

\section{Preliminaries}\label{sec:prelim}

In this section, we introduce the relevant notation and some preliminary lemmas.  Throughout this paper, we will view $\mathbb F_2$ as a rooted tree. Therefore, our results will cover self-quasi-isometries of rooted trees. 

\begin{definition}
Let $(X, d)$ be a metric space, we say that a map $f: X\to X$ is a \emph{$C$-quasi-isometric embedding} for a constant $C\geq 1$ if
\begin{align*}
    \frac{d(x, y)}{C} - C\leq d(f(x), f(y)) \leq C d(x, y) + C.
\end{align*}
for all $x, y\in X$.

Further, we say that a $C$-quasi-isometric embedding $f: X\to X$ is a \emph{$C$-quasi-isometry} if there exists a constant $D$ such that for all $y\in X$ there there exists $x\in X$ such that $d(y, f(x))\leq D$. 
\end{definition}

\begin{definition} Let $(X, d)$ be a metric space. Two maps $f, g : X\to X$ are \emph{$C$-bounded} if $d(f(x), g(x))\leq C$ for all $x\in X$. They are \emph{bounded} if they are $C$-bounded for some constant $C$.
\end{definition}

\subsection{Notation on trees}

Let $T$ be a rooted tree and $w\in T$ a vertex. We assume throughout that trees have edge length exactly $1$. We denote the subtree rooted at $w$ by $T_w$. Further the subtree $T_{w}^k\subset T_w$ is the induced subtree of all vertices $v\in T_w$ with $d(w, v)\leq k$. Vertices $v\in T_w$ are called \emph{descendants} of $w$ and $w$ is called an \emph{ancestor} of $v$. Further, a vertex $v\in T_w$ is a \emph{$D$-child} of $w$ if $d(v, w) = D$ and we say that $w$ is the \emph{$D$-parent} of $v$. We denote the ($1$-)parent of a vertex $v\in T$ by $p(v)$ and say that the parent of the root is itself.

We will view a path between vertices $u$ and $v$ as a sequence of neighbouring vertices $u=u_0,u_1, \ldots, u_n=v$, denoted by $(u_0, \ldots, u_n)$. If a path $(u_0, \ldots, u_n)$ is geodesic (or equivalently non-backtracking) we also denote it by $[u_0, u_n]$.

\begin{definition}
For a subset $U \subseteq T$ of a rooted tree $T$ based at $v_0$, we define the \emph{lowest common ancestor} of $U$ as the (unique) vertex $v\in T$ furthest away from $v_0$ such that every vertex $u\in U$ is a descendant of $v$. We will denote this vertex $v$ as $LCA(U)$.
\end{definition}

Observe that if $v = LCA(U)$, then there exists a pair of vertices $x, y\in U$ such that $v$ lies on $[x, y]$.

\begin{definition}
    Let $S$ be a finite subtree of a rooted tree $T$. We say that the \emph{boundary} of $S$, denoted by $\partial S$, is the set of vertices $v\in T\setminus S$ whose parent $p(v)$ is in $S$.
\end{definition}

\begin{remark}
\label{rmk:isoperimetry_tree}
    If $T$ is a $d$-regular tree rooted at $v_0$, then one can easily show by induction that $\abs{\partial S} = \abs{S}(d-2) +1$ if $v_0\not\in S$ and  $\abs{\partial S} = \abs{S}(d-2) +2$ if $v_0\in S$. 
\end{remark}

\begin{definition}
Let $T$ be a tree rooted at $v_0$. A map $f: T\to T$ is \emph{order-preserving} if for every pair of vertices $u, v\in T$ with $v\in T_u$ we have that $f(v)\in T_{f(u)}$.
\end{definition}

In \cite{Nairne}, Nairne shows that every $(1,C)$-quasi-isometry between spherically homogeneous trees is at bounded distance from an order-preserving quasi-isometry. In Lemma \ref{lemma:close_to_growingQI} we extend this result and show that any $C$-quasi-isometry of a rooted tree to itself is at bounded distance from an order-preserving quasi-isometry. 

\subsection{Properties of quasi-isometries of trees} We state and prove three key technical lemmas about quasi-isometries of trees.

The following lemma states that the image of the geodesic $[u, v]$ under a quasi-isometry $f$ coarsely surjects onto the geodesic $[f(u), f(v)]$.  

\begin{lemma}\label{lemma:close_to_geodesic}
Let $T$ be a tree and let $f: T\to T$ be a $C$-quasi-isometry. For every pair of vertices $u, v\in  T$ and vertex $a\in [f(u), f(v)]$ there exists a vertex $b\in [u, v]$ such that $d(f(b), a)\leq C$.
\end{lemma}

\begin{figure}[h]
    \centering
    \includegraphics[width=.6\textwidth]{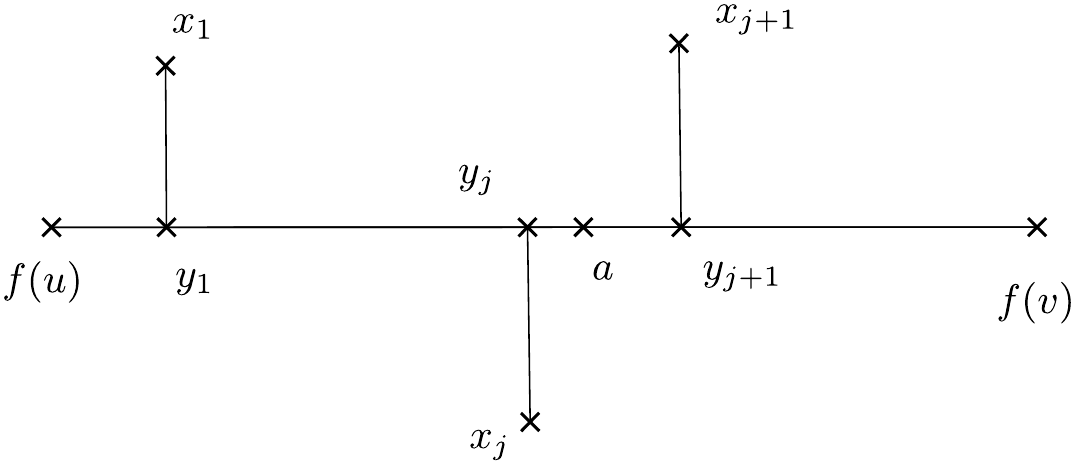}
    \caption{Images of geodesics coarsely surject onto the geodesic.}
    \label{fig:path}
\end{figure}

\begin{proof}
Let $[u, v] = (u_0, \ldots , u_n)$. For $0\leq i \leq n$, define $x_i = f(u_i)$ and let $y_i$ be the closest point projection of $x_i$ onto $[f(u), f(v)]$. This is depicted in Figure \ref{fig:path}. 
Let $j$ be the largest index such that $y_j\in [f(u), a]$. Then the path $[x_j, y_j][y_j, y_{j+1}][y_{j+1}, x_{j+1}]$ is non-backtracking and hence a geodesic from $x_j$ to $x_{j+1}$ going through $a$. Since $f$ is a $C$-quasi-isometry, $ d(x_j, a)+d(a, x_{j+1}) = d(x_j ,x_{j+1}) \leq 2C$. So $\min\{d(x_j, a)+d(a, x_{j+1})\}\leq C$. 
\end{proof}

The following lemma states that every quasi-isometry between a rooted tree and itself is at bounded distance from an order-preserving quasi-isometry. This extends the result of \cite{Nairne} where this is shown for $(1, C)$-quasi-isometries between spherically homogeneous trees.

\begin{lemma}\label{lemma:close_to_growingQI}
Let $T$ be a tree rooted at $v_0$ and let $f:T \to T$ be a $C$-quasi-isometry. The map $f$ is at bounded distance from an order-preserving quasi-isometry. Moreover, if $f(v_0) = v_0$, then $f$ is at $K$-bounded distance from an order-preserving $(2K+C)$-quasi-isometry for some $K$ depending only on $C$.
\end{lemma}

\begin{figure}
    \centering
    \includegraphics{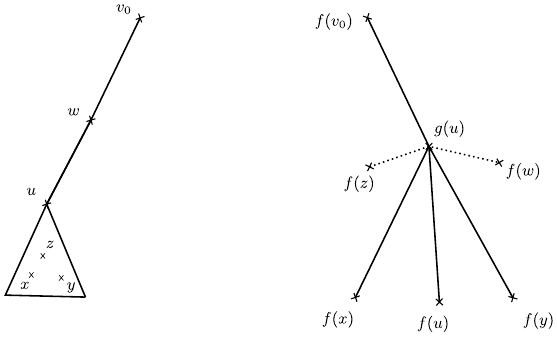}
    \caption{Quasi-isometries are at bounded distance from order-preserving quasi-isometries.}
    \label{fig:close_to_growing}
\end{figure}
\begin{proof}
It suffices to show the moreover part with $K =  3C^3+2C$. Define $g: T\to T$ via $g(v) := LCA(f(T_v))$. Clearly, $g$ is order-preserving. It remains to show that $g$ is at $K$-bounded distance from $f$ since it then follows that $g$ is a $(2K+C)$-quasi-isometry .

Let $u\in T$ be a vertex, we will show that $d(f(u),g(u)) \leq K$. We have $f(u)\in T_{g(u)}$, thus by Lemma \ref{lemma:close_to_geodesic}, there exists $w\in [v_0, u]$ such that $d(f(w), g(u))\leq C$. This is depicted in Figure \ref{fig:close_to_growing}. Since $g(u) = LCA(f(T_u))$, there exist vertices $x, y\in T_u$ such that $g(u)\in [f(x), f(y)]$.  Again by Lemma \ref{lemma:close_to_geodesic}, there exists a vertex $z\in [x, y]\subset T_u$ with $d(g(u), f(z))\leq C$. In particular, $d(f(w), f(z))\leq 2C$.

Observe that $u\in [w, z]$. Hence, $d(u, z)\leq d(w, z)\leq 3C^2$. Therefore, $d(g(u), f(u))\leq d(g(u), f(z))+d(f(z), f(u))\leq 3C^3+2C = K$.
\end{proof}

The following lemma states that if $f$ is an order-preserving quasi-isometry and two vertices $u, v$ have the same distance from the root $f(u)$ cannot be a descendant of $f(v)$, unless they are close. This lemma is a key ingredient in the proof of Lemma \ref{lemma:every_qi_is_mixed}.

\begin{lemma}\label{lemma:same_dist_properties}
Let $T$ be a tree rooted at $v_0$ and let $f: T \to T$ be an order-preserving $C$-quasi-isometry. Let $u, v\in T$ be vertices such that $d(v_0, u) = d(v_0, v)$ and $f(u)\in T_{f(v)}$. Then $d(f(u), f(v))\leq K$ and $d(u, v)\leq K$ for some constant $K$ depending only on $C$.
\end{lemma}

\begin{figure}[h]
    \centering
    \includegraphics[width=.4\textwidth]{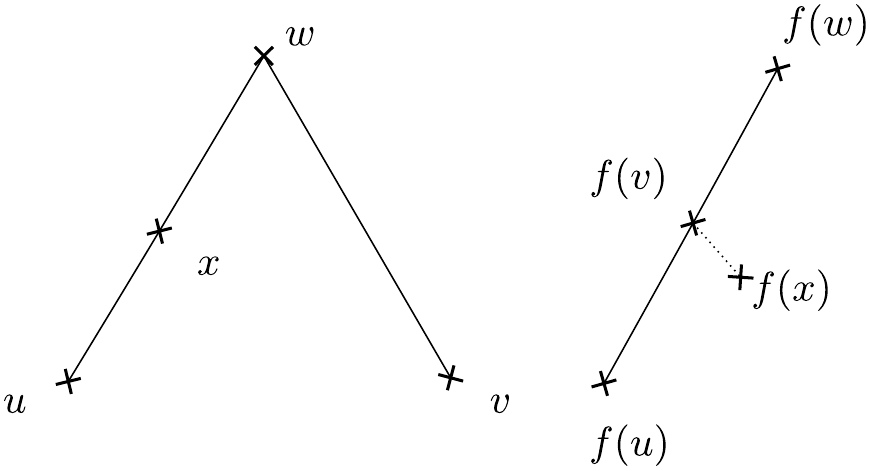}
    \caption{Illustration of the proof of Lemma \ref{lemma:same_dist_properties}.}
    \label{fig:same_height}
\end{figure}

\begin{proof}
Let $ w = LCA(\{u, v\})$. Since $f$ is order-preserving, $f(v)$ lies on $[f(u), f(w)]$. This is depicted in Figure \ref{fig:same_height}. By Lemma \ref{lemma:close_to_geodesic}, there exists a vertex $x\in [u,w]$ such that $d(f(x), f(v))\leq C$. Thus $d(w, v)\leq d(x,v )\leq 2C^2$. Since $d(v_0, u) = d(v_0, v)$, $d(u, v) = 2d(w, v)$ and hence $d(f(u), f(v))\leq 4C^3 + C$. So choosing $K = 4C^3 + C$ works.
\end{proof}

\section{Quasi-isometries of regular trees}\label{sec:qi_of_reg_trees}
\textbf{Notation:} For the rest of this section, $T$ denotes a regular tree of degree $d \geq 3$ rooted at a vertex $v_0$.\\

In this section, we describe a way of building quasi-isometries, which we call \emph{mixed-subtree quasi-isometries}, of regular trees to themselves. We further show that any quasi-isometry is at bounded distance from a mixed-subtree quasi-isometry. The key idea behind mixed-subtree quasi-isometries is that they are quasi-isometries which are defined iteratively for vertices further and further away from the root. Moreover, at each step, the allowed choices are in some sense independent from the choices for earlier vertices.

\textbf{Construction:} Let $D \geq 1$ be a natural number. For all natural numbers $i\geq 0$ we inductively construct functions $f_i : T_{v_0}^{iD}\to T$. Define $f_0(v_0) = v_0$. Assuming we have defined $f_{i}$, we define $f_{i+1}$ as follows. 

\begin{figure}
    \centering
    \includegraphics{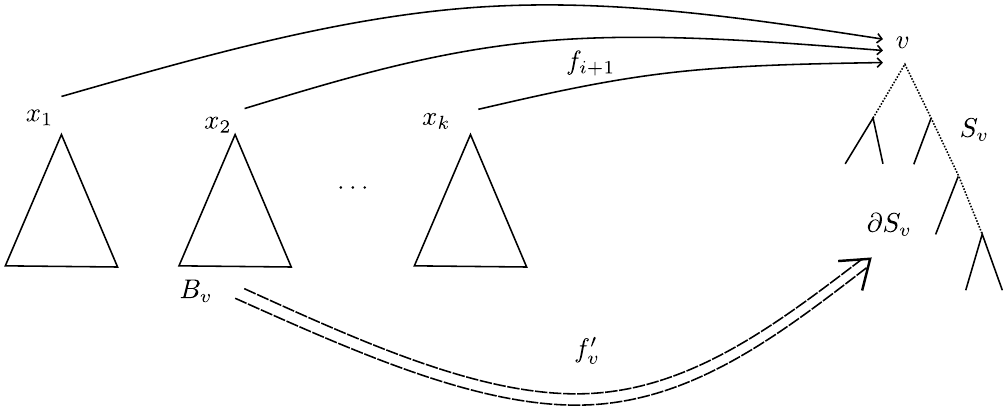}
    \caption{Definition of $f'$}
    \label{fig:mixed_subtree}
\end{figure}

\begin{itemize}
     \item For all vertices $x\in T_{v_0}^{iD}$ define $f_{i+1}(x) = f_i(x)$.
     \item Iterate through all vertices $x\in T$ with $d(v_0, x) = iD$. If we have not yet defined $f_{i+1}$ for any descendants of $x$, do the following. 
     \begin{itemize}
        \item Denote $f_{i}(x)$ by $v$ and let $X  = \{x_1,\ldots x_k\}$ be the set of vertices that satisfy $f_i(x_j) = v$ and $d(v_0, x_j) = iD$. Define $B_v$ as the set of all $D$-children of vertices $x_j\in X$. We now define $f_{i+1}(h)$ for all vertices $h\in B_v$. 
        \item Choose any function $f_v':B_v \to T_v$ satisfying the following properties, see Figure \ref{fig:mixed_subtree}.
        \begin{enumerate}
                \item\label{item:image_is_boundary_subtree}  $\mathrm{Im}(f_v')=\partial S_v$ for some finite subtree $S_v$ of $T_v$ containing $v$. 
                \item\label{item:kind_of_injective}  If $f_v'(w) = f_v'(w')$, then $w$ and $w'$ are $D$-children of the same vertex $x_j\in X$.
        \end{enumerate}
        \item Define $f_{i+1}|_{B_v} = f_v'$.
        \item For all $x_j\in X$, define $f_{i+1}(w) = v$ for all vertices $w\in T_{x_j}^{D-1}$.
     \end{itemize}
\end{itemize}

We first argue that there always exists at least one function $f_v'$ satisfying \eqref{item:image_is_boundary_subtree} and \eqref{item:kind_of_injective}. In other words, we have to show that there exists a subtree $S_v$ rooted at $v$ such that $\abs{X}\leq \abs{\partial S_v}\leq \abs{B_v}$. If $D = 1$ and $\abs{X} = 1$, then one can choose $S_v = \{v\}$ to get $\abs{B_v} = \abs{\partial S_v}$. Otherwise $\abs{B_v} - \abs{X}\geq d - 1$ hence by Remark \ref{rmk:isoperimetry_tree} we can find a subtree $S_v$ rooted at $v$ with $\abs{X}\leq \abs{\partial S_v}\leq \abs{B_v}$.
 
Further note that with this definition for every $i, j\in \mathbb{N}$, $f_i$ and $f_j$ agree if they are both defined. Hence we can define $f: T\to T$ via $f(v) = f_i(v)$ for some $i$ where $v$ is in the domain of $f_i$. We call any map $f$ constructed this way a \emph{$D$-deep mixed-subtree} quasi-isometry.

The following lemma shows that mixed-subtree quasi-isometries are indeed quasi-isometries. 

\begin{lemma}\label{lemma:is_qi}
For any choice of functions $f_v'$, the map $f$ constructed is an order-preserving $C$-quasi-isometry, where $C$ only depends on $D$ and $T$.
\end{lemma}

\begin{proof}
    It follows directly from the definition that $f$ is order-preserving. Let $K = d^D$, where $d$ is the degree of $T$ and let $C = 2 K^2$. We will show that $f$ is a $C$-quasi-isometry.

    \begin{claim}\label{Claim 1}
        If vertices $b, b'$ are $D(i+1)$-children of $v_0$, then $f(b)\neq f(b')$ unless the $D$-parents of $b$ and $b'$ are the same. Furthermore, if $f(b)\neq f(b')$, then $T_{f(b)}$ and $T_{f(b')}$ are disjoint.
    \end{claim}
    \textit{Proof of Claim \ref{Claim 1}.} We prove this by induction on $i$. For $i=0$, $b$ and $b'$ have the same $D$-parent, namely $v_0$. The furthermore part follows from \eqref{item:image_is_boundary_subtree}. Assume the statement is true for $i$, we want to show that it holds for $i+1$. Let $x$ and $x'$ be the $D$-parents of $b$ and $b'$ respectively. If $f(x)\neq f(x')$, then $T_{f(x)}$ is disjoint from $T_{f(x')}$ by the induction hypothesis. Hence $f(b)\neq f(b')$ and $T_{f(b)}$ is disjoint from $T_{f(b')}$. If $f(x) = f(x')$ and $x\neq x'$, then \eqref{item:kind_of_injective} implies that $f(b)\neq f(b')$. Moreover,  \eqref{item:image_is_boundary_subtree} implies that $T_{f(b)}$ and $T_{f(b')}$ are disjoint. Lastly, if $x = x'$, we only have to show the furthermore part, which follows from \eqref{item:image_is_boundary_subtree}. \hfill$\blacksquare$

    \begin{claim}\label{Claim 2}
        For any $i$, the number of $Di$-children of $v_0$ whose image under $f$ coincide is at most $K$.
    \end{claim}
    \textit{Proof of Claim \ref{Claim 2}.} This follows from Claim \ref{Claim 1} together with the fact that every vertex has at most $K$ $D$-children. \hfill$\blacksquare$

    \begin{claim}\label{Claim 3}
        If $b$ is the $D$-child of a vertex $x$ which in turn is a $Di$-child of $v_0$, then $1\leq d(f(b), f(x))\leq K^2$.
    \end{claim}
    \textit{Proof of Claim \ref{Claim 3}.} Let $v = f(x)$. We use the notation from the construction of $f_{i+1}$. By Claim \ref{Claim 2}, the set $B_v$ contains at most $K^2$ vertices, so $\abs{\mathrm{Im}(f_v')}\leq K^2$. In other words the subtree $S_v$ from \eqref{item:image_is_boundary_subtree} has at most $K^2$ leaves, implying that $d(v, v')\leq K^2$ for any vertex $v'\in \partial S_v$ (see Remark \ref{rmk:isoperimetry_tree}). Consequently $d(f_v'(b), f(x))\leq K^2$, which concludes the proof. \hfill$\blacksquare$

    \begin{claim}\label{Claim 4}
        The map $f$ is $K^2$-coarsely surjective.
    \end{claim}
     \textit{Proof of Claim \ref{Claim 4}.} 
     First observe that whenever a vertex $v$ is in the image of $f$, there exists a $Di$-child $x$ of $v_0$ with $f(x) = v$. 
     
     Let $v'\in T$ be a vertex. We show that $d(v', \mathrm{Im}(f))\leq K^2$. Let $v$ be the lowest ancestor of $v'$ which is in the image of $f$. We have that $v = f(x)$ for some vertex $x$ which is a $Di$ child of $v_0$. If $v' = v$, we are done. If $v'\in S_v$, then $d(v, v')\leq K^2$ as in the proof of Claim \ref{Claim 3}. If $v'\not\in S_v$, there exists $w\in \partial S_v$ which is a descendant of $v$ and an ancestor of $v'$. Since $w\in \partial S_v$, it is in the image of $f$, a contradiction with the definition of $v$.  \hfill$\blacksquare$

    It remains to show that 
    \begin{align*}
        \frac{d(u, v)}{C} - C\leq d(f(u), f(v))\leq C d(u, v) + C
    \end{align*}
    for all vertices $u, v\in T$. To show the right half of the inequality, it is enough to show that for all neighbours $u, v\in T$, we have $d(f(u), f(v))\leq C$. This follows directly from the definition of $f$ and Claim \ref{Claim 3}. Next we show the left half of the inequality. Let $u, v\in T$ be vertices and let $n = \floor{d(v_0, u)/D}, m = \floor{d(v_0, v)/D}$. Define $u_0 = v_0$ and for $i\leq n$ define $u_i$ as the $Di$-child of $v_0$ which is an ancestor of $u$. Define $v_i$ analogously. Let $k$ be the maximal index such that $u_k = v_k$. Claim \ref{Claim 1} together with $f$ being order-preserving yield that $f(u_{i})$ and $f(v_j)$ lie on the geodesic from $f(u)$ to $f(v)$ for all $k+2\leq i \leq n$ and $k+2\leq j \leq m$. Hence, $d(f(u), f(v))\geq (n - k - 2)+(m - k-2)$. On the other hand $d(u, v) \leq D(n - k +1) + D(m - k +1 )$. The statement follows. 
\end{proof}

We are now ready to prove the following lemma which together with Lemma \ref{lemma:is_qi} states that a map $g : T\to T$ is a quasi-isometry if and only if it is at bounded distance from a mixed-subtree quasi-isometry. The lemma is a slightly more detailed version of Theorem \ref{thm:main}.

\begin{lemma}\label{lemma:every_qi_is_mixed}
Let $g : T\to T$ be a $C$-quasi-isometry. There exists a constant $D>0$ and a $D$-deep mixed subtree quasi-isometry $f$ such that $g$ and $f$ are at bounded distance. Moreover, if $g(v_0) = v_0$, then $D$ only depends on $T$ and $C$.
\end{lemma}
\begin{proof}
By Lemma \ref{lemma:close_to_growingQI}, which states that all quasi-isometries are at bounded distance from order-preserving quasi-isometries, it suffices to show the moreover part for an order-preserving quasi-isometry. So we assume in the following that $g$ is order-preserving. 

Let $K$ be the constant of Lemma \ref{lemma:same_dist_properties} and let $D = \ceil{C(C+K)+1}$. We will show that there is a $D$-deep mixed subtree quasi-isometry $f$ at distance $K+CD+C$ from $g$.

Assume that we have defined $f_i: T_{v_0}^{iD} \rightarrow T$, as in the construction, such that
\begin{enumerate}[label = \roman*)]
    \item $d(f_i(u), g(u))\leq K$, for all $u$ with $d(v_0, u) = Di$,\label{cond:boundedhyp1}
    \item $g(u)\in T_{f_i(u)}$, for all $u$ with $d(v_0, u) = Di$,\label{cond:subsethyp}
    \item $d(f_i(w), g(w))\leq K + CD + C$ for all $w\in T_{v_0}^{iD}$. \label{cond:bound2}
\end{enumerate}
We show that we can define a function $f_{i+1}$ such that
\begin{enumerate}
    \item[a)] $d(f_{i+1}(u), g(u))\leq K$, for all $u$ with $d(v_0, u) = D(i+1)$, \label{cond:boundedind}
    \item[b)] $g(u)\in T_{f_{i+1}(u)}$, for all $u$ with $d(v_0, u) = D(i+1)$, \label{cond:subsetind1}
    \item[c)] $d(f_{i+1}(w), g(w))\leq K + CD + C$ for all $w\in T_{v_0}^{(i+1)D}$.
\end{enumerate}

Let $x$ be a $Di$ child of $v_0$, let $v = f_i(x)$ and let $X = f_i^{-1}(v)$. Observe that for all $x'\in X$, $d(v_0, x') = Di$. Let $B_v$ be the set of all $D$-children of elements of $X$ and let $A_v = g(B_v)$. By \ref{cond:subsethyp},  $A_v\subset T_v$. For $b\in B_v$, define $f_v'(b)$ as the vertex $a\in A_v$ closest to $v_0$ which satisfies $g(b)\in T_{a}$. Observe that $g(b)\in T_{f_v'(b)}$, in other words, b) is satisfied.

Note that $f_v'(b) = g(b')$ for some $b'\in B_v$. It follows from Lemma \ref{lemma:same_dist_properties} that $d(f_v'(b), g(b))\leq K$ for all $b\in B_v$, which proves a). Therefore, $g|_{B_v}$ and $f_v'$ are at $K$-bounded distance. By \ref{cond:boundedhyp1}, $d(f_i(x'), g(x'))\leq K$ for all $x'\in X$. Hence for a $k$-child $w$ of some $x'\in X$ for $k< D$ we have $f_{i+1}(w) = v$ and hence $d(f_{i+1}(w), g(w))\leq d(v, g(x')) + d(g(x'), g(w))\leq K+CD+C$, which, together with \ref{cond:bound2}, proves c). 

It only remains to show that $f_v'$ as defined above is a valid choice, that is, $f_v'$ satisfies \eqref{item:image_is_boundary_subtree} and \eqref{item:kind_of_injective}. For \eqref{item:image_is_boundary_subtree}, define $S_v = \{y \in T_v| \text{$y\not\in T_a $ for all $a\in A_v$}\}$. If $w\in \partial S_v$, then $w\in T_{a_w}$ for some $a_w\in A_v$ while its parent is not in $T_{a_w}$. It follows that $w = a_w$. Further, for any $b\in g^{-1}(a_w)$ we have that $f_v'(b) \in [a_w, v_0]$ but $f_v'(b)\not\in S_v$. Hence $f_v'(b) = a_w$ implying that $a_w \in \mathrm{Im} (f_v')$. 

Thus $\partial S_v = \mathrm{Im}(f_v')$ is finite. If $S_v$ is infinite, there exists a vertex $u\in S_v$ which is further away from $v_0$ than all points in the finite set $\partial S_v$. Consequently, $T_u\subseteq S_v$. Since $g$ is a quasi-isometry (and hence coarsely surjective), there exists a vertex $u'\in T$ with $d(v_0, u')\geq (i+1)D$ and $g(u')\in T_u$. We have that $u'$ is the descendant of some $Di$-child $x'$ of $v_0$. By Claim \ref{Claim 1} from the proof of Lemma \ref{lemma:is_qi} either $x'\in X$ or $T_{f_i(x')}$ is disjoint from $T_v$. By \ref{cond:subsethyp} the latter cannot be the case. Consequently, $u'\in T_b$ for some $b\in B_v$ and since $g$ is order preserving, $g(u')\in T_{g(b)}$. This is a contradiction to $g(u')\in S_v$. Thus $S_v$ is indeed finite.

In order to prove \eqref{item:image_is_boundary_subtree}, it remains to show that $v\in S_v$, or in other words, that $v\not\in A_v$. Let $b\in B_v$ be a $D$-child of some vertex $x'\in X$. By \ref{cond:boundedhyp1} and the fact that $g$ is a $C$-quasi-isometry, $d(g(b), v)\geq d(g(b), g(x')) - d(g(x'), f_i(x')) \geq  d(g(b), g(x')) - K >0$, so indeed $g(b)\neq v$. Since this is true for all $b\in B_v$, it follows that $v\not\in A_v$.

Next we prove \eqref{item:kind_of_injective}. Let $b$ be a $D$-child of $x$ and $b'$ be a $D$-child of $x'$ with $x\neq x'\in X$. We have $d(b, b')\geq 2D$. Thus $d(g(b), g(b'))\geq 2D/C - C > 2K+C$, which implies that $d(f_v'(b), f_v'(b')) > C$. In particular, $f_v'(b)\neq f_v'(b')$.
\end{proof}

\bibliography{main}
\bibliographystyle{alpha}

\end{document}